\documentclass[12pt]{amsart}

\usepackage[utf8]{inputenc}
\usepackage{amsmath,amsthm,amssymb,amsfonts,euscript,graphicx,hyperref,indentfirst}
\usepackage{amssymb}
\usepackage{amsbsy}
\usepackage{graphicx}

\calclayout

\newtheorem{theorem}{Theorem}[section]
\newtheorem{definition}{Definition}
\newtheorem{corollary}{Corollary}[section]
\newtheorem{lemma}{Lemma}[theorem]
\newtheorem{question}{Question}

\begin{document}
\title[Confined subgroups of Thompson's group $F$] {Confined subgroups of Thompson's group $F$ and its embeddings into wobbling groups}
\author{Maksym Chaudkhari}
\begin{abstract}
	We obtain a characterisation of confined subgroups of Thompson's group $F$. As a result, we deduce that orbital graph of a point under action of $F$ has uniformly subexponential growth if and only if this point is fixed by the commutator subgroup.  This allows us to prove non-embeddability of $F$ into wobbling groups  of graphs with uniformly subexponential growth.  
\end{abstract}
\maketitle
\section{Introduction}
  Discovered in 1965, Richard Thompson's group $F$ became an important object of study for geometric group theory and measured group  theory. This group  has an unusual combination of properties $-$ it is a finitely presented, torsion-free group, which has a free subsemigroup but no free subgroups, and its commutator subgroup is simple. Furthermore, amenability of $F$ remains a major open question. Recall that an action of a group $G$ on a set $X$ is amenable if  $X$ carries a $G$-invariant finitely additive probability measure. A group is called amenable if its action on itself by left multiplication is amenable. The class of amenable groups contains all finite and abelian groups, and a group is called elementary amenable if it could be obtained from finite and abelian groups by taking extensions, direct limits, quotients or subgroups. Thompson's group $F$  is known to be not elementary amenable. We refer the reader to \cite{Jose} and \cite{CFP} for more details about Thompson's group $F$.\\
  One of the approaches to proving  amenability of a group is based on the study of its embeddings into wobbling group of a suitable graph. In particular,  amenability of the  full topological group of a Cantor minimal system was established using certain  embedding into $W(\mathbb{Z})$ and the fact that action of $W(\mathbb{Z})$ on $\mathbb{Z}$ is extensively amenable (notion first defined in \cite{Kate+Vova}),  see  \cite{Kate} or \cite{Kate+Monod}. Furthemore, amenability of interval exchange transformation group, or shortly the IET, is equivalent to  a question about amenability of subgroups of wobbling group $W(\mathbb{Z}^d)$ for $d>3$, we refer the reader to \cite{Kate+} for more details.
 Therefore, it is natural to try to determine whether  $F$ could be embedded into a wobbling group of a "small" graph. In this case small would mean that its wobbling group's action is extensively amenable. This condition is satisfied, for example, for recurrent graphs. Furthermore, at this time there are no known examples of graphs with polynomial or uniformly subexponential growth for which the condition fails (see \cite{Kate+}).
 
 In this article we   prove that $F$ does not embed into a wobbling group of a graph with uniformly subexponential growth. We follow the approach  developed by Nicolas Matte Bon in  \cite{Nico} for full groups of minimal etale grupoids of germs. In particular, we  obtain a characterisation of confined subgroups of Thompson's group $F$ which is equivalent to a characterisation of its Schreier graphs which do not contain arbitrarily large balls isomorphic to a ball in the Cayley graph of $F$. Since standard Cayley graph of $F$ is extremely hard to visualize,  this result gives us some characterisation of Schreier graphs of $F$ which one could hope to construct.  We would like to notice that the study of Schreier graphs of $F$ has already provided  valuable information about this group. For example, a  description of a family of  Schreier graphs of $F$ by Dmytro Savchuk in \cite{Savchuk} is used in study of  Poisson-Furstenberg boundary  of $F$, see \cite{Pasha},\cite{Kaimanovich}, \cite{K+T} and recent preprints \cite{Stankov} and \cite{Lamenable}.
 \subsection*{Acknowledgements}
 The author is grateful to his advisor  Kate Juschenko for introducing him to the topic and her guidance, support and encouragement. The author would also like to thank Nicolas Matte Bon for   valuable remarks  about a preliminary version of this article.
\section{Preliminaries}
\subsection{ Properties of Thompson's group F}
This subsection contains a brief overview of  properties of Thompson's group $F$  which are used in this article. Proofs of majority of these facts could be found in \cite{Jose} or \cite{CFP}.
 \begin{definition}
 	Thomposon's group $F$ is a group formed by all orientation-preserving piecewise linear  homeomorhpisms of segment $[0,1]$ which have   breaking points in dyadic rationals  with slopes equal to powers of $2$. 
 \end{definition} 
In this article we consider right action of $F$ on unit segment,  thus for any two homeomorphisms $f, g \in F$ their composition is defined as $hg(x)= g(h(x))$. It is well-known that for any $n\geq 1$ action  of $F$ on ordered $n-$tuples of dyadic rationals is transitive. The \textbf{support} of   $g \in F$ is  the closure of the set of points  of unit segment on which $g$ acts nontrivially.  Thompson's group $F$ has finite presentation: $$\langle x_0,  x_1 | [x_0x_1^{-1},x_0^{-1}x_1x_0]=[x_0x_1^{-1},x_0^{-2}x_1x_0^2]=id \rangle, $$ where the generators are defined as follows:
$$ x_0(t) = \begin{cases} t/2, & 0 \leq t \leq 1/2 \\ t-1/4, & 1/2 \leq t\leq 3/4\\ 2t-1, & 3/4 \leq t \leq 1 \end{cases}, \mbox{ and }  x_1(t)= \begin{cases} t, & 0\leq t\leq 1/2 \\ t/2+1/4, & 1/2 \leq t \leq 3/4 \\ t-1/8, & 3/4 \leq t \leq 7/8\\ 2t-1, & 7/8 \leq t \leq 1 
\end{cases}.
$$ Its commutator subgroup $F'$ is a simple group which coincides with the subgroup of all elements  with support contained in the unit interval $(0,1$). The quotient  $F/F'$ is isomorphic to $\mathbb{Z}^2$, and a subrgoup  $H<F$ is normal if and only if it contains the commutator subgroup of $F$.\\
 For any segment   $[a,b] \subseteq [0,1]$ (or interval $(a,b)$) denote $F[a,b]$ (respectively $F(a,b)$) the subgroup of all elements of $F$ with support in $[a,b]$ (respectively $(a,b)$). We will need  the following properties of these subgroups:
 \begin{enumerate}
 	\item If $a<b$ are dyadic rationals,  then $F[a,b]$ is isomorphic to $F$. Furthermore, its commutant is exactly the subgroup $F(a,b)$.
 	\item If $a<b<c<d$ are dyadic rationals  then group generated by elements of $F[a,c]$ and $F[b,d]$ contains $F[a,d]$
  \end{enumerate}

\subsection{ Schreier graphs and orbital graphs}
\begin{definition}
	Assume that $G$ is a group generated by a finite set $S$ and let $H<G$ be its subgroup. The Schreier graph of $G$ modulo $H$ is an oriented labelled graph with the set of vertices  equal to the set of right cosets $\{Hg, g \in G\}$ and the set of edges equal to $\{(Hg, Hgs), s \in S\}$. We denote $\Gamma(G,H)$ the Schreier graph of $G$ modulo $H$.
\end{definition} 
If $G$ acts on a set $X$ and $p\in X$,   the Schreier graph of $G$ modulo  stabilizer of $p$ is called \textbf{the orbital graph} of $p$. Its vertex set can be identified with the orbit $\mathcal{O}_G(p)$ of $p$, and two vertices $v,w \in\mathcal{O}_G(p)$ are  connected by edge with label $s$ if and only if $s(v)=w$.  \\

For a connected graph with bounded degree one can  define its uniform growth function as follows:
\begin{definition}
Let $\Gamma=(V,E)$ be a connected graph with bounded degree. For  $v \in \Gamma$   denote $B_\Gamma (v,n)=$ a ball of radius $n$ centered at $v$ in $\Gamma$.  Then uniform growth function  of $\Gamma$ is defined as 
$$\bar{b}(n) = \sup_{v \in V} |B_\Gamma(v,n)|,$$
where $|B_\Gamma(v,n)|$ stands for cardinality of the set of vertices. 
\end{definition}
If $\Gamma$ is a Schreier graph of G, changing a finite generating set of $G$ preserves equivalence class of uniform growth function of $\Gamma$ under the following equivalence relation:\\ 
For two functions $f, g:\mathbb{N} \rightarrow \mathbb{N},$  $f$ is said to grow asympotically  not slower than $g$ ($g \preceq f$) if there exists a constant $C>0$ such that $g(n) < f(Cn)$ for all $n \in \mathbb{N}$ and $f \sim g$ if and only if  $g \preceq f$ and  $f \preceq g$. \\
In general, this equivalnce class is preserved under bi-Lipschtitz embeddings and we  refer to this class when we talk about uniform growth rate.
One should note that although uniform growth is always greater or equal to usual growth function of graph, they might be completely different and it is easy to construct an example of graph with exponential uniform growth which has linear growth function.\\
We also consider Schreier  graph $\Gamma(G,H)$  as a rooted graph with root at $H$. Space of rooted labelled graphs could be naturally equipped with a distance $d$ defined as follows. Let $\Gamma_1$ and $\Gamma_2$ be two rooted labelled graphs with roots $v_1$ and $v_2$. Then $d(\Gamma_1,\Gamma_2)=1/{n+1}$, where $n \geq 0$ is the least integer such that $B_{\Gamma_1}(v_1,n)$ and $B_{\Gamma_2}(v_2,n)$ are not isomorphic as rooted labelled graphs.
\subsection{Chabauty space} 
Let $G$ be a countable discrete group.  The set of its subgroups $Sub(G)$ is  endowed with topology ( called \textit{Chabauty topology}) induced from the space $2^G$ of all subsets of $G$ with usual product topology. Base of this topology  is formed by sets $$U_{A,B}=\{ H \in Sub(G): A \subset H, H \cap B = \emptyset \},$$ where $A,B$ are finite subsets of $G$. With Chabauty topology $Sub(G)$ becomes a compact metrizable space on which $G$ acts by conjugation and this is an action by homeomorphisms. \\
 Assume that $G$ is finitely generated and fix its generating set $S$. Convergence in Chabauty topology has a reformulation in terms of Schreier graphs. Namely, a sequence of subgroups $H_{n}$ converges to $H<G$ if and only if, for any generating set $S$, the sequence $\Gamma(G,H_n)$ converges to $\Gamma(G,H)$ in the space of labelled rooted graphs. \\
Suppose that  $H$ and $K$ are subgroups of $G$, then $H$ is said to be \textbf{confined by $K$} if the  closure of $K$-orbit  of $H$ does  not contain  the trivial subgroup.  Since the base of neighbourhoods of trivial subgroup is formed by sets of form $U_{\{1\},P}$ with $P$ finite, the last condition is equivalent to the existence of a finite set $P=\{g_1,g_2,...,g_r \} \subseteq G\setminus{1}$, such that $\forall k \in K:$ $$kHk^{-1} \cap P \neq \emptyset, $$ we  call such sets $P$ confining. 
In case   when $K=G$ group  $H$ is simply called \textbf{confined}. \\In terms of Schreier graphs, $H$ is not confined  if and only if one can find a sequence of vertices $v_n, n\geq 1$, in the Schreier graph $\Gamma(G,H)$, such that versions of $\Gamma(G,H)$ rooted at $v_n$  converge to the Cayley graph of $G$ in the space of rooted labelled graphs.\\
Confined subgroups are also related to study of \textit{uniformly recurrent subgroups} - closed minimal $G$-invariant subsets of $Sub(G)$ (it is easy to see that any nontrivial uniformly recurrent subgroup  consists of confined subgroups). We refer the reader to \cite{Ardien&Nico} for description of uniformly recurrent subgroups of Thompson's groups and its applications to $C^*$-simplicity.
\subsection{Wobbling groups}
The wobbling groups were first studied in \cite{Laczko} and applied to Tarski's circle-squaring problem.
\begin{definition}
	Let $\Gamma =(V,E)$ be a locally finite  connected graph equipped with standard graph metric $d_\Gamma$.  The wobbling group $W(\Gamma)$ is defined as a group of all bijections $g:V\rightarrow V$ such that $$\sup_{x \in V} d_\Gamma(x,g(x)) < \infty$$
\end{definition}
One can show that the wobbling group of any graph containing infinite path  contains a free subgroup, but no property (T) group could be embedded into a wobbling group of a graph with uniformly subexponential growth, see \cite{Kate} or \cite{Kate+Mikael}. However, as it was pointed out by Matte Bon, if one removes uniformity requirement, any residually finite group  can be embedded into a wobling group of graph of linear growth. In particular, group $SL_3 (\mathbb{Z})$ which has property (T) embeds into a wobbling group of graph with linear growth. Therefore, general embeddability questions would require some uniformity, although in case of Thompson's group F, which has few normal subgroups, question without assumptions about uniformity still makes sense.\\
Finally, observe that if $G$ is a finitely generated group with generating set $S$, and $H$ is its subgroup, then action of $G$ on right cosets of $H$ defines a homomorphism from $G$ to $W(\Gamma(G,H))$.
\section{Characterisation of confined subgroups of Thompson's group F} 
In this section we obtain a characterisation of subgroups of $F$ confined by its commutator subgroup $F'$ in terms of the action of $F$ on unit interval. This characterisation is analogous to one obtained in Theorem 4.1 of \cite{Nico}, although the proof of Theorem 4.1   does not directly apply to $F$, because in the present case corresponding  action of a confined subgroup on Cantor space may have infinite fixed closed proper subsets, and as result  statement of step $1$ in\cite{Nico} is false in this setting.\\
 If a group $G$ acts by homeomorphisms on a tolopological space $X$ and $Y \subset X$, we call a \textit{rigid stabilizer} of $Y$ a subgroup of $G$ consisting of all elements which  fix the complement of $Y$ pointwise. We denote this subgroup $R_G (Y)$, and stabilizer of $Y$ is denoted $St_G(Y)$.  Finally, a subgroup  of all elements $g \in G$ which act trivially on some  neighborhood of $Y$ is denoted $St^{0}_G(Y)$ and is called the germ stabilizier of $S$.
\begin{theorem}\label{thrm1}
	A subgroup $H$ of Thompson's group $F$ is confined by the commutator subgroup $F'$ of $F$ if and only if there exists a finite subset of unit segment $S \subset [0,1]$ such that $ St^{0}_{F'}(S)\leq H\leq St_{F}(S).$ In particular, a subgroup is confined if and only if it is confined by $F'$.
\end{theorem}
\begin{proof}
	We first prove that any subgroup satisfying inclusions from the theorem is confined. It suffices to show that for any finite $S$ the subgroup $St^0_{F'}(S)$ is confined by $F'$. Put $r=|S| +1$ and take any nontrivial $g_1,g_2,...,g_r \in F$ with pairwise disjoint supports. Then for any $h \in F'$ support of at least one of the elements $h^{-1}g_1h,h^{-1}g_2h,...,h^{-1}g_rh$ does not intersect $S$, and thus this element belongs to the germ stabilizer of $S$.
	 
	To prove the reverse direction, consider a maximal open subset $V$ of interval $(0,1)$ such that  $H$ contains every $g \in F$, whose support belongs to $V$. It is easy to see that such maximal set exists, since if $H$  contains every element supported in one of a family of open sets, then it contains every element supported in their union.  Our aim is to show that the complement of $V$ is finite. We first  check that $V$ is non-empty. We are going to use the following theorem (see also a similar theorem 3.10 in \cite{Ardien&Nico})
	\begin{theorem}[Nicolas Matte Bon, \cite{Nico}] Let $G$ be a countable group acting by homeomorphisms of a Hausdorff space $X$, and assume that $A \leq G$ is a subgroup whose action on X is minimal and proximal. Let $H \in Sub(G)$ be confined by $A$. Then there exists a non-empty open subset $U \subset X$ and a finite index subgroup $\Gamma$ of rigid stabilizer  $R_A(U)$ such that $H$ contains the derived subgroup $[\Gamma,\Gamma]$.
    \end{theorem}
 Apply this theorem to $X=(0,1)$, $A=F'$ and $H$ (action of $F'$ is transitive on ordered n-tuples of dyadic rationals for any $n$, so it is minimal and proximal on unit interval). Let $a<b$ be dyadic rationals such that $[a,b]$  is contained in $U$. As we mentioned before, subgroup $F[a,b]$ of elements of $F$ supported on $[a,b]$ is isomorphic to Thompson's group F and its commutator is simple and coincides with the group $F(a,b)$ of all elements of F supported on interval $(a,b)$. Then since $F'[a,b]$ is simple, it must be contained in $\Gamma$ and consequently in its derived subgroup. Therefore $H$ contains  $F(a,b)$ and $V$ is non-empty.\\
Next, we show that  $V$ has only finitely many connected components. Suppose that $P=\{g_1,g_2,...,g_r \}$ is confining for $H$. Note that if we replace some elements of $P$ with their inverses and reorder elements of $P$, we will still have a confining set. We  need the following fact:
\begin{lemma}\label{clm:tec}
	For any nontrivial $ g_1,g_2,...,g_r \in F,$ possibly after permuting and taking inverses, one can find  intervals $U_1,...,U_r$ with dyadic endpoins such that  $g_1(U_1) < U_1<g_2(U_2)<U(2) <...<g_r(U_r)<U_r$, where interval $(a,b)$ is less than $(c,d)$ if $b<c$.
\end{lemma}
\begin{proof}
	We induct on $r$. Case $r=1$ is obvious. For the inductive step, choose an element with maximal supremum of support. Take this element as $g_r$, and let $s$ be the supremum of its support.  Note that $s$ must be a fixed point of  ${g_r}$. If $g_r$ is greater than indentity on   $(s-\epsilon, s]$ for any sufficiently small $\epsilon$, take it's inverse. It remains to apply inductive hypothesis to elements $ g_1,g_2,...,g_{r-1} $  and then  choose $U_r$ sufficiently close to $s$ to ensure that desired inequalities hold for $U_r$.
\end{proof}
Suppose that one can find $r$ connected components of $V$:          
 $(x_1, y_1), (x_2, y_2),...,(x_r, y_r)$,  such that $0<x_1$ and $ y_i\leq x_{i+1}$ for $i=1,...,r-1$. Using  lemma \ref{clm:tec}, we can show that at least one of these intervals could be extened over its left endpoint contradicting definition of connected component. Indeed,  let $U_i$, $i=1,...,r$, be as in lemma \ref{clm:tec} and let $g_1,g_2,..,g_r$ be corresponding modification of $P$ , then one can choose dyadic intervals $V_i \subset (x_i,y_i)$ and sufficiently small dyadic intervals $ W_i, x_i \in W_i,$ $i=1,...,r$, such that $W_1 <V_1 <W_2<V_2<...<W_r<V_r$. Since $F'$ acts transitively on increasing $4r$-tuples of dyadic rationals, there exists $h \in F'$ such that $h(V_i)=U_i, h(W_i)=g_i(U_i), i=1,...,r$. Then for some $i$ an element $ k= h^{-1}g_ih$ belongs to  $H$  and $k(V_i)=W_i$. Consequently, $F'[W_i]=k^{-1}F'[V_i]k < H$. Therefore any element of $F$ with support in $W_i$ belongs to $H$. Now, since elements with support in  $W_i$ together with elements with support in $(x_i,y_i)$  generate group of all elements supported on $ W_i\cup (x_i,y_i)$, we obtain an interval in $V$ which contains  $(x_i,y_i)$ as a strictly smaller subinterval, a contradiction. As a result,  $V$ can not have more than $r+1$ connected components.\\
Since $V$ consists of finitely many intervals, endpoins of these intervals must be fixed by $H$, and $[0,1] \setminus V$  is a disjoint union of  finitely many  segments and points. Assume that segment $[x,y], x<y,$ is a connected component of the complement of $V$, then it's endpoints must be fixed by $H$. Notice that lemma \ref{clm:tec}  together with transitivity on r-tuples  imply that $H$ can not have more than $r$ fixed points in the interval $(0,1)$, so $H$ can not act on $[x,y]$ trivially. Consider element $h \in F'$ which maps all breaking points (except for $0$ and $1$) of  each of $ g_1,g_2,...,g_r$   inside $(x,y)$. Then $h^{-1}Ph=\{ h^{-1}g_1h, h^{-1}g_2h,...,h^{-1}g_rh \}$ is still a confining set for $H$. Furthermore, each of its elements either has support in $(x,y)$ or moves at least one endpoint of $[x,y]$. Since $h^{-1}Ph$  is confining, its conjugation by any element of $F'$ must hit $H$. Thus, if we conjugate $h^{-1}Ph$ by elements with support in $(x,y)$ , we will still be hitting $H$, but elements of $h^{-1}Ph$ which move endpoint of this segment will still move the endpoint, so they can  not belong to $H$. Therefore, we can only consider those elements of $h^{-1}Ph$ which are supported on $[x,y]$. But this implies that restriction of $H$ to $[x,y]$ is confined by subgroup of all elements of F with support on $(x,y)$. Then we can repeat argument in the beginning of the proof to obtain  a subinterval $(z,t) \subset (x,y)$ that must belong to $V$, which contradicts definition of $[x,y]$. Consequently, complement of $V$ is some finite set of points $S$ which are fixed by $H$. Thus, by definition of $V$, 
$St^{0}_{F'}(S)\leq H\leq St_{F}(S)$, which completes the proof of the theorem.
\end{proof}
\subsection*{Remark 1.} Theorem \ref{thrm1} implies that graph $\Gamma(F,H)$, with $H$ being a confined subgroup of $F$, must be amenable. Indeed, it sufficies to prove this for $H=St^0_{F'}(S)=St^0_{F}(S \cup \{ 0,1\})$ with $S \cap (0,1) \neq \emptyset$. Notice that  $x_0(t)< t, \forall t \in (0,1)$, so for any finite $S \subset (0,1)$ there exists $k \in \mathbb{N}$ such that $x_0^k(S) \subset (0,1/2)$. Then, since $x_1(t)=t$   and $x_0(t)=t/2 $ on $(0,1/2)$ and they coincide in a neighbourhood of 1, $\Gamma(F,H)$ contains  arbitrarily large parts of square grid and thus it is amenable. 
\subsection*{Remark 2.} The theorem above also implies that the subgroups  constructed in \cite{Gili} are  examples of maximal non-confined subgroups of  $F$. 

\section{Growth of orbital graphs of Thompson's group F}
Results from previous section allow us to deduce lower bounds on uniform growth of orbital graphs of $F$  in the same fashion as in section $5$ of  \cite{Nico}.
\begin{theorem}
	Assume that $F$ acts on a set $X$ and let $p$ be any point in $X$. Then either orbital graph of $p$ has exponential uniform growth or it is fixed by commutator subgroup of $F$. 
\end{theorem}
\begin{proof}
	Orbital graph of $p$ is isomorphic to the Schreier graph of $St_F(p)$. Assume that $St_F(p)$ is not confined by $F'$. Then there exists a sequence of points $p_n, n\geq 1,$ in orbit of $p$ such that orbital graphs of $p_n$ converge to Cayley graph of $F$ in space of rooted labelled graphs. Then, as $F$ has exponential growth, orbital graph of $p$  has exponential uniform growth. \\
	If $St_F(p)$ is confined, then it either contains commutator subgroup $F'$ or it fixes a point $x$ in unit interval $(0,1)$. In the latter case  orbital graph  of $p$ grows not slower than  Schreier graph of $x$. But according to classification of Schreier graphs of points from unit interval obtained by  Savchuk in \cite{Savchuk}, all these graphs have exponential growth, which completes the proof.
\end{proof}
\begin{corollary}
	$F$ and $F'$ do not embed into wobbling groups of graphs with uniformly subexponential growth. 
\end{corollary}
\begin{proof}
	It suffices to notice that the orbital graph of a point under the action of a finitely generated subgroup of wobbling group has uniform growth not exceeding uniform growth of the initial graph. Then, if F acts on a graph by elements of its wobbling group,   either every point of a graph is fixed by commutator subgroup or it has exponential uniform growth. Conclusion for  $F'$ follows from the fact that $F$ embeds in $F'$. 
\end{proof}

\section{Final remarks and questions }
 We still do not know whether $F$ could be embedded into a wobbling group of a recurrent graph with bounded degree  or into a wobbling group of a  graph with subexponential growth. In the  latter case, arguments used in section $4$ might fail only for a point $p$ with non-confined stabilizer if a sequence $p_n, n\geq 1,$ is sufficiently sparse in the orbital graph of $p$. Similarly, for recurrent graphs the case of confined subgroup follows from results of Savchuk and Kaimanovich or Mischenko.  For non-confined subgroup $H$ the fact  that arbitrarily large balls from the Cayley graph appear in the Schreier graph modulo $H$ does not imply that the Schreier graph is not recurrent, since one can modify a recurrent graph by inserting large components of Cayley graph of $F$ without affecting its recurrence.  
 The following question still remains open: 
 \begin{question}
Is there any non-confined maximal subgroup of $F$ which corresponds to recurrent Schreier graph? In particular, are the Schreier graphs  modulo subgroups defined in \cite{Gili}  always transient?  
\end{question}
 It also would be natural to try to estimate the density of fragments of the Cayley graph of $F$ in a Schreier graph modulo its non-confined subgroup, possibly with additional assumptions concerning amenability. Affirmative answer to the following question would obviously settle the general case of graphs with subexponetial growth.  \\
 \begin{question}
 	Is it true that for any non-confined subgroup $H<F$ one can find constant $K$ such that for infinitely many $n \in \mathbb{N}$ there is a  copy of a ball $B_F(n)$ in a ball of the Schreier graph with radius $Kn$ centered at the root?
 \end{question}

\bigskip
Department of Mathematics, Northwestern University; 2033 Sheridan Road, Evanston,
IL 60208, USA.\\
\textit{E-mail address:} maxc@math.northwestern.edu

\begin{thebibliography}{20}

	
	\bibitem{Jose} J. Burillo. Introduction to Thompson's group F. Book published online: \url{https://mat-web.upc.edu/people/pep.burillo/F\%20book.pdf}
	\bibitem{CFP} J.W. Cannon, W.J.Floyd and W.R. Parry. Introductory notes on Richard Thompson's groups. Enseign. Math. (2), 42(3-4),215–256, 1996.
	\bibitem{Gili} G. Golan, M. Sapir. On subgroups of R. Thompson's group F. Transactions of the American Mathematical Society, Volume 369, 8857-8878, 2017.
	
	\bibitem{Kate} K. Juschenko. Amenability of discrete subgroups by examples. Book in preparation: \url{https://sites.math.northwestern.edu/~juschenk/files/BOOK.pdf}

   	\bibitem{Kate+} K. Juschenko, N. Matte Bon, N. Monod, M. de la Salle. Extensive amenability and an application to interval exchanges. Ergodic Theory and Dynamical Systems 38, no. 1, 195–219, 2018.
	\bibitem{Kate+Monod} K. Juschenko, N. Monod. Cantor systems, piecewise translations and simple amenable groups. Annals of Mathematics, Volume 178, Issue 2, 775 - 787, 2013. 
	\bibitem{Kate+Vova}K. Juschenko, V. Nekrashevych, M. de la Salle. Extensions of amenable groups by recurrent groupoids. Inventiones mathematicae, Volume 206, Issue 3, 837-867, 2016.
	\bibitem{Kate+Mikael}K. Juschenko, M. de la Salle. Invariant means for the wobbling group. Bulletin of the Belgian Mathematical Society Simon Stevin, Volume 22, no. 2, 281-290, 2015.
	\bibitem{K+T}K. Juschenko, T. Zheng. Infinitely supported Liouville measures of Schreier graphs. Preprint, arXiv:1608.03554.
	\bibitem{Lamenable} K. Juschenko. A remark on Liouville property of strongly transitive actions. Preprint, arXiv:1806.02753.
	\bibitem{Kaimanovich} V. Kaimanovich. Thompson's group is not Liouville. In T. Ceccherini-Silberstein, M. Salvatori, E. Sava-Huss (Eds.), Groups, Graphs and Random Walks. London Mathematical Society Lecture Note Series,300-342. Cambridge University Press, Cambridge, 2017.
	\bibitem{Laczko} M. Laczkovich. Equidecomposability and discrepancy; a solution to Tarski's circle-squaring problem. Journal für die reine und angewandte Mathematik, Issue 404, 77-117, 1990.
	\bibitem{Ardien&Nico} A. Le Boudec, N. Matte Bon. Subgroup dynamics and $C^*$-simplicity of  groups of homeomorphisms. Annales scientifiques de l'ENS 51, fascicule 3, 557-602, 2018.
	\bibitem{Nico} N. Matte Bon. Rigidity of graphs of germs and homomorphisms between full groups. Preprint, arXiv:1801.10133. 
	\bibitem{Pasha} P. Mishchenko. Boundary of the action of Thompson's group $F$ on dyadic numbers. Preprint, arXiv:1512.03083. 
	\bibitem{Savchuk} D. Savchuk. Schreier graphs of actions of Thompson's group F on the unit interval and on the Cantor set. Geometriae Dedicata, V.175, 355–372, 2015.
    \bibitem{Stankov} B. Stankov. Non-triviality of the Poisson boundary of random walks on the group $H(Z)$ of Monod. Preprint, arXiv:1806.00301 
 
\end{thebibliography}
\end{document}